\documentclass[11pt]{article}

\usepackage[margin=1in]{geometry}
\usepackage[T1]{fontenc}
\usepackage{lmodern}
\usepackage[hidelinks]{hyperref}
\usepackage{tikz}
\usepackage{microtype}
\usepackage{amsmath,amssymb,amsfonts,amsthm}
\usepackage{mathtools}
\usepackage{bm}
\usepackage{graphicx}
\usepackage{booktabs}
\usepackage{enumitem}

\theoremstyle{plain}
\newtheorem{theorem}{Theorem}[section]
\newtheorem{proposition}[theorem]{Proposition}
\newtheorem{lemma}[theorem]{Lemma}
\newtheorem{corollary}[theorem]{Corollary}

\theoremstyle{definition}

\theoremstyle{remark}

\title{Higher-Order Derivatives Do Not Accelerate the\\ Computation of Fixed Points}

\author{
    Uijeong Jang\thanks{Department of Mathematics, University of California, Los Angeles ({uijeongjang@math.ucla.edu})}
    \and
    Ernest K. Ryu\thanks{Department of Mathematics, University of California, Los Angeles ({eryu@math.ucla.edu})}
}

\date{}

\begin{document}
\maketitle

\begin{abstract}
The Picard iteration converges to the unique fixed point of a $q$-contractive operator at a linear rate $q^N$, and a lower bound with an affine construction shows that no deterministic method querying only operator values can do better. But what about higher-order methods that query derivatives? A single Jacobian evaluation reveals an affine map entirely, so the affine construction says nothing about higher-order methods. In this work, we show that finite-order derivative information still does not accelerate the worst-case complexity for smooth contractive fixed-point computation. This contrasts with higher-order smooth minimization, where higher-order derivatives do improve worst-case rates for convex and non-convex minimization.

\end{abstract}

\section{Introduction}

Consider the fixed-point problem
\[
\begin{array}{ll}
{\mbox{find}} & x_\star \in \mathbb R^n \\
{\mbox{such that}} & F(x_\star)=x_\star,
\end{array}
\]
where $F\colon\mathbb R^n\to\mathbb R^n$ is a contraction mapping with contraction factor $q\in(0,1)$.
Then the classical Picard iteration
\cite{banach1922operations,picard1890memoire}
\[
    x_{t+1}=F(x_t),\qquad t=0,1,\ldots,
\]
converges linearly:
\begin{equation}\label{eq:picardupperbound}
     \|x_t-x_\star\|\le q^t\|x_0-x_\star\|.
\end{equation}
This elementary estimate is one of the most basic and foundational convergence results in numerical analysis and optimization \cite{ortega2000iterative,kelley1995iterative,bauschke2017convex}.

A natural question is whether Picard iteration can be accelerated.  For strictly contractive mappings, if a method is allowed to query only the value $F(x)$, a lower bound based on an affine construction $F(x)=Ax+b$ establishes that the Picard rate cannot be improved  \cite{nemirovsky1983problem,nesterov2018lectures,park2022exact}. However, if an oracle both returns the operator value and first derivative, then one query yields
\[
    DF(x)=A,
\]
thereby providing all the information needed to recover the fixed point at once. Consequently, the affine lower bound says nothing about the complexity of higher-order methods for fixed-point computation. It is therefore natural to ask whether access to higher-order derivative information can accelerate the computation of fixed points.

In this paper, we show that, in fact, higher-order derivatives do not make acceleration possible. Specifically, we prove that for every fixed finite order $p\in \mathbb{N}$, having access to 
\[
\{F(x),DF(x),D^2F(x),\dots , D^pF(x)\}
\]
does not improve the worst-case fixed-point error and fixed-point residual for smooth contractive fixed-point problems. This stands in contrast
to smooth minimization, where higher-order information does improve worst-case rates; we return
to this comparison in Section~\ref{ss:sketch} and Section~\ref{ss:prior-work}.
% \cite{arjevani2019oracle, carmon2020lower}.

\paragraph{Contributions.} 
We prove exact lower bounds for deterministic finite-order methods. For a fixed finite oracle order $p$, a contraction factor $0<q<1$, an iteration/query budget $N\ge 1$, a deterministic method that queries the value and derivatives of $F$ up to order $p$ with a starting point $x_0$, and $\varepsilon\in(0,1)$, there is a $q$-contractive map $F$ with a unique fixed point $x_\star$ such that the method's $N$-th iterate $x_N$ satisfies
\[
    \|x_N-x_\star\|\ge (1-\varepsilon)q^N\|x_0-x_\star\|.
\]
Therefore, Picard iteration is exactly worst-case optimal for fixed-point error over the class of $q$-contractive mappings: no deterministic
finite-order method can improve even the leading constant in the rate
$q^N$ (Theorem~\ref{thm:error} and Corollary~\ref{cor:minimaxerror}). Also, when performance is measured by the fixed-point residual $\|x_N-F(x_N)\|$, similar proof technique also yields a matching optimality result:
\[
    \|x_N-F(x_N)\|\ge (1-\varepsilon)(1+q)\left(\sum_{j=0}^{N}q^{-j}\right)^{-1}\|x_0-x_\star\|.
\]
This matches the exact value-oracle residual constant attained by OC-Halpern~\cite{park2022exact} (Theorem~\ref{thm:residual} and Corollary~\ref{cor:minimaxerrorresidual}). 
Also, by taking $q\to 1$, we recover the matching $\Theta(1/N)$ bound for nonexpansive operators~\cite{lieder2021convergence,contreras2023optimal,park2022exact} (Corollary~\ref{cor:minimaxerrorresidual}).

\subsection{Proof sketch}\label{ss:sketch}
We briefly outline the core proof technique. We first construct a one-dimensional auxiliary function $\varphi\colon\mathbb{R}\rightarrow\mathbb{R}$. It is $p$-times continuously differentiable and satisfies 
\[ 
    0 \le \varphi' \le q,  \qquad  \mathrm{Lip}(\varphi^{(p)}) \le L, \qquad \varphi^{(k)}(0)=0 \textup{ for } 0\le k\le p,
\]
and is affine on the tail. Figure~\ref{fig:phi} illustrates the shape of this function.  

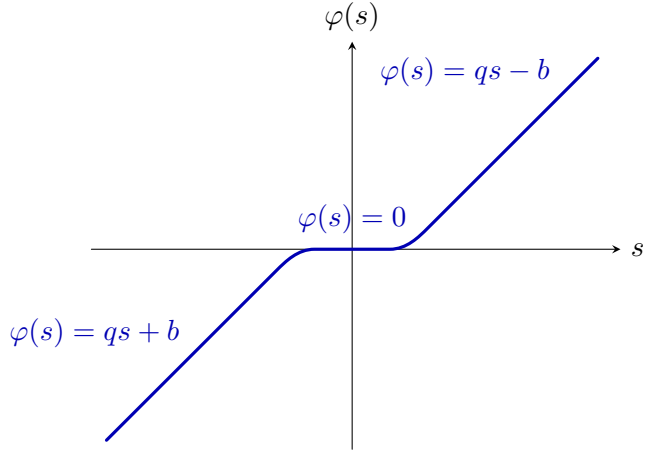
\begin{figure}[h]
\centering
\begin{tikzpicture}[
    x=1.0cm,y=1.0cm,
    >=stealth,
    line cap=round,
    line join=round
]

    \def\B{0.72}

    \draw[->] (-3.45,0) -- (3.55,0) node[right] {$s$};
    \draw[->] (0,-2.65) -- (0,2.75) node[above] {$\varphi(s)$};

    \draw[very thick, blue!70!black]
        (-3.25,-2.53) -- (-1,-0.28)
        .. controls (-0.88,-0.16) and (-0.72,0) .. (-0.5,0)
        -- (0.5,0)
        .. controls (0.72,0) and (0.88,0.16) .. (1,0.28)
        -- (3.25,2.53);

    \node[blue!70!black, above] at (0,0.08) {$\varphi(s)=0$};
    \node[blue!70!black, above left] at (-2.15,-1.43) {$\varphi(s)=qs+b$};
    \node[blue!70!black, above left] at (2.75,2.03) {$\varphi(s)=qs-b$};
\end{tikzpicture}
\caption{Plot of the auxiliary scalar function $\varphi$, which is flat near the origin and affine on the tail.}
\label{fig:phi}
\end{figure}

Then we consider the following operator, where $v_1,\dots,v_m\in\mathbb{R}^n$ are orthonormal:
\begin{equation}\label{eq:map}
    F(x)=c v_1+\sum_{i=1}^{m-1}\varphi(\langle v_i,x\rangle)v_{i+1}.
\end{equation}
Assume, without loss of generality, the method starts from the origin, i.e., $x_0=0$. Then, the oracle information up to the $p$-th derivative reveals only the direction $v_1$ at the first iteration. At iteration $1<t\le N$, the oracle can reveal directions in $\mathrm{span}\{v_1,\dots,v_{t+1}\}$. This chain argument can be made more formal with the resisting-oracle technique \cite{nemirovsky1983problem}.

The fixed point is placed on the positive affine branch of $\varphi$ by choosing 
\[ 
s_1=c,\qquad qs_i-b=s_{i+1}, \qquad x_\star=\sum_{i=1}^{m}s_i v_i.
\] 
Then $\varphi(s_i)=s_{i+1}$ and hence $F(x_\star)=x_\star$. After $N$ queries, the final output $x_N$ is orthogonal to the hidden tail $v_{N+1},\dots,v_m$. The distance from $x_N$ to the hidden tail of $x_\star$ then gives the non-trivial lower bound
\[
    \|x_N-x_\star\| \ge (1-\varepsilon)q^N\|x_0-x_\star\|
\]
for any $\varepsilon\in(0,1)$.

When the performance is measured by the fixed-point residual, we modify the operator $F$:
\[
F(x)=c v_1-\varphi(\langle v_{N+1},x\rangle)v_1 +\sum_{i=1}^{N}\varphi(\langle v_i,x\rangle)v_{i+1}.
\]
The fixed point is again placed on the positive affine branch of $\varphi$ and the direction $v_{N+1}$ remains hidden. This yields
\[
\|x_N-F(x_N)\| \ge(1-\varepsilon)(1+q)
\left(\sum_{j=0}^{N}q^{-j}\right)^{-1} \|x_0-x_\star\|
\]
for every $\varepsilon\in (0,1)$. 

This no-acceleration lower bound contrasts with smooth minimization, where
higher-order derivative information does improve the optimal rate. Indeed, a
convex minimization problem can be equivalently written as
\[
\begin{array}{ll}
{\mbox{find}} & x_\star \in \mathbb R^n \\
{\mbox{such that}} & \nabla f(x_\star)=0,
\end{array}
\]
but the corresponding vector field $F\colon=\nabla f$ has more constraints. Namely, $F=\nabla f$ it is the
gradient of a scalar potential, so its Jacobian $DF(x)=\nabla^2 f(x)$ must be
symmetric at every $x\in\mathbb{R}^n$. Contractive operators, in contrast, are
subject to no such symmetry constraint, and this  extra freedom enables the
adversarial chain \eqref{eq:map}, whose Jacobian is strictly lower triangular
in the basis $v_1,\ldots,v_m$ and hence far from symmetric, to fully hide the
structural information contained in future coordinates from any finite-order
oracle.

\subsection{Prior work}
\label{ss:prior-work}

\paragraph{Oracle complexity with higher-order derivatives.}
Oracle complexity studies how many oracle queries are needed to reach a target accuracy for a given problem, and has been a standard framework for understanding the limitations of optimization algorithms. In convex optimization, the standard accuracy criterion is the function-value gap $f(x_N)-f_\star\le \varepsilon$. For first-order methods over $L$-smooth convex functions, Nesterov acceleration reaches this accuracy in $O(\varepsilon^{-1/2})$ gradient queries, matching the classical lower bounds of Nemirovsky and Yudin~\cite{nemirovsky1983problem,nesterov2018lectures}. In the $L$-smooth, $\mu$-strongly convex case, the corresponding optimal complexity is $O(\sqrt{\kappa}\log (1/\varepsilon))$, with condition number $\kappa=L/\mu$. Exact-information-complexity results sharpen these statements to exact constants for several smooth convex classes \cite{drori2017exact,drori2022oracle,jang2025computer,chari2026optimal}. For higher-order methods with $p\ge 2$, accelerated $p$-th-order tensor methods reach $\varepsilon$-accuracy in $O(\varepsilon^{-2/(3p+1)})$ $p$-th-order oracle queries \cite{kovalev2022first}. This complexity matches known lower bounds for higher-order smooth convex minimization \cite{arjevani2019oracle}. These function-value lower bounds are often built from affine or quadratic chain constructions: the oracle reveals one new direction at a time, while the unobserved tail carries the remaining objective gap.

\paragraph{Oracle complexity for nonconvex stationarity.}
A parallel line of work studies lower bounds for finding approximate $\varepsilon$-stationary points, defined as points satisfying $\|\nabla f(x)\|\le \varepsilon$, in smooth nonconvex optimization. With only Lipschitz
gradients, gradient descent achieves
$O(\varepsilon^{-2})$ first-order oracle complexity, and this rate is worst-case optimal \cite{carmon2020lower}. Higher-order smoothness plays two distinct roles in smooth nonconvex optimization. First, one may assume higher-order smoothness of the objective while still giving the method only first-order information. In this setting, first-order methods can improve over gradient descent. For example, accelerated first-order methods achieve rates of order $\varepsilon^{-7/4}$ under Lipschitz Hessians and
$\varepsilon^{-5/3}$ under Lipschitz third derivatives, up to logarithmic
factors in some variants \cite{carmon2017convex,jin2018accelerated,carmon2018accelerated, li2023restarted}. Zero-chain lower bounds showed that such improvements are nevertheless limited: deterministic first-order methods require $\Omega(\varepsilon^{-12/7})$ queries in the Hessian-Lipschitz setting and $\Omega(\varepsilon^{-8/5})$ queries for arbitrarily smooth objectives \cite{carmon2021lower}. A very recent work closes these exponent gaps, proving matching deterministic first-order lower bounds of $\Omega(\varepsilon^{-7/4})$ in the Hessian-Lipschitz case and $\Omega(\varepsilon^{-5/3})$ under Lipschitz third derivatives~\cite{zhou2026sharp}.

Second, one may strengthen the oracle itself by allowing higher-order derivative queries. Even then, for functions with Lipschitz $p$-th derivatives, any randomized algorithm requires $\Omega(\varepsilon^{-(p+1)/p})$ queries to find an $\varepsilon$-stationary point~\cite{carmon2020lower}. This matches the known upper bounds for gradient descent, cubic-regularized Newton methods, and more general $p$-th order regularization methods in their corresponding smoothness classes~\cite{nesterov2006cubic,cartis2011adaptive,nesterov2018lectures}.

\paragraph{Fixed-point oracle complexity.}
The convergence theory of contractive fixed-point iteration goes back to Banach's contraction principle and its numerical-analysis consequences \cite{banach1922operations,ortega2000iterative,kelley1995iterative}. The oracle-complexity viewpoint has also been used to show that Picard's $q^N$ rate is worst-case optimal: for value-oracle, or zeroth-order, fixed-point methods, the standard lower-bound argument shows that the rate in the Picard fixed-point error bound cannot be improved for strict contractions \cite{nemirovsky1983problem}. On the other hand, the standard performance measure in fixed-point problems is often the fixed-point residual $\|x-Tx\|$. In this setting, Halpern-type methods attain an $O(1/N)$ residual rate \cite{sabach2017first,lieder2021convergence,kim2021accelerated,diakonikolas2020halpern,contreras2023optimal}, which is known to be optimal for nonexpansive fixed-point problems \cite{park2022exact}. These results are fixed-point analogs of optimization lower bounds. However, existing fixed-point results do not address whether finite-order derivative information can improve the worst-case complexity of smooth contractive and nonexpansive fixed-point problems.

\subsection{Notation and definitions}
Throughout this paper, all spaces are finite-dimensional Euclidean spaces.  We write $\langle \cdot,\cdot\rangle$ and $\|\cdot\|$ for the Euclidean inner product and norm.  The norm of a linear map is the induced operator norm, and the norm of a multilinear map is the corresponding operator norm.  Write $\mathbb N = \{1,2,3,\ldots\}$ for the set of positive integers.

For a $p$-times continuously differentiable map, or a $C^p$ map, $F\colon\mathbb R^n\to\mathbb R^n$ where $p\in\mathbb N$, we denote by $D^kF(x)$ its $k$-th derivative at $x$, viewed for $k\ge1$ as a symmetric $k$-linear map.  We also use the convention
\[
    D^0F(x)=F(x).
\]
Thus the derivatives of $F$ up to order $p$ at $x$ are
\[
    \bigl\{D^kF(x)\bigr\}_{k=0}^p.
\]
We write $\mathrm{Lip}(D^pF)$ for the Lipschitz constant of the map $x\mapsto D^pF(x)$, i.e.,
\[
    \mathrm{Lip}(D^pF) := \sup_{x\neq y}\frac{\|D^pF(x)-D^pF(y)\|}{\|x-y\|},
\]
where the norm on $D^pF(x)-D^pF(y)$ is the operator norm on $p$-linear maps.

A map $F\colon\mathbb R^n\to\mathbb R^n$ is called a \emph{$q$-contraction} if
\[
    \|F(x)-F(y)\|\le q\|x-y\|\qquad\text{for all }x,y\in\mathbb R^n,
\]
where $q\in(0,1)$.  A map $T\colon\mathbb R^n\to\mathbb R^n$ is called \emph{nonexpansive} if
\[
    \|T(x)-T(y)\|\le \|x-y\|\qquad\text{for all }x,y\in\mathbb R^n.
\]
If $x_\star=F(x_\star)$, we call $x_\star$ a fixed point of $F$.

For $p\in\mathbb N$, $q\in(0,1)$, and $L>0$, define $\mathcal C_{p,q,L}(\mathbb R^n)$ to be the class of maps $F\colon\mathbb R^n\to\mathbb R^n$ such that
\[
    F\in C^{p},\qquad \|DF(x)\|\le q\quad\text{for all }x\in\mathbb R^n,
\]
and
\[
    \|D^pF(x)-D^pF(y)\|\le L\|x-y\|\qquad\text{for all }x,y\in\mathbb R^n.
\]
Every map in $\mathcal C_{p,q,L}(\mathbb R^n)$ with $q\in(0,1)$ is a $q$-contraction and hence has a unique fixed point. 

Finally, a \emph{deterministic $p$-th order method} is a sequence of deterministic update rules $\{A_t\}_{t\ge 1}$ which, given an initial point $x_0\in\mathbb R^n$ and oracle access to a map $F$, produces iterates $x_1,x_2,\ldots$ such that
\[
    x_t=A_{t}\left(x_0,\{D^kF(x_s)\}_{k=0}^p,\ 0\le s<t\right), \quad t\ge1.
\]
That is, each query point may depend arbitrarily on all previously observed values and derivatives
up to order $p$, but on no other information about $F$.

\paragraph{Organization.} Section~\ref{sec:chain} proves the exact fixed-point-error lower bound. Section~\ref{sec:residual} proves the exact residual lower bound for contractions and, in the limit $q \to 1$, the $\Omega(1/N)$ residual lower
bound for nonexpansive maps. Section~\ref{sec:conclusion} concludes.

\section{Fixed-point lower bound for contractive maps}\label{sec:chain}

We first state the formal exact fixed-point-error lower bound theorem.

\begin{theorem}\label{thm:error} 
Let $p,N\in\mathbb N$, $q\in(0,1)$, $L>0$, and $\varepsilon\in(0,1)$. Then there exists
\[
n_0=n_0(p,q,N,\varepsilon)\in\mathbb N
\]
with the following property: For every $n\ge n_0$, every starting point $x_0\in\mathbb R^n$, and every deterministic $p$-th order method initialized at $x_0$, there exists a map $F\in\mathcal C_{p,q,L}(\mathbb R^n)$ with a unique fixed point $x_\star$ such that the method's $N$-th iterate $x_N$ satisfies
\[
\|x_N-x_\star\|\ge (1-\varepsilon)q^N\|x_0-x_\star\|.
\] 
\end{theorem} 
Since Picard iteration guarantees $\|x_N-x_\star\|\le q^N\|x_0-x_\star\|$ for every $q$-contraction, and since $\varepsilon\in(0,1)$ in Theorem~\ref{thm:error} is arbitrary, the theorem pins down the exact deterministic worst-case complexity.

\begin{corollary}\label{cor:minimaxerror}
No deterministic method can improve upon Picard's factor $q^N$ in worst-case error after $N$ oracle queries, even when it is allowed to query all derivatives of $F$ up to any fixed finite order $p$.
\end{corollary}

The rest of this section constructs the hard instance used to prove Theorem~\ref{thm:error}.  The proof is by a smooth finite-order chain construction, which is best viewed as a smooth finite-order analog of the chain constructions used in oracle lower bounds~\cite{nemirovsky1983problem,arjevani2017oracle,drori2022oracle,carmon2021lower,zhou2026sharp}. 

\subsection{An auxiliary scalar function}
We first construct a one-dimensional scalar function illustrated in Figure~\ref{fig:phi}. It behaves like the affine map on the tail, while all of its derivatives up to order $p$ vanish at the origin. This function will later be used to prevent a finite-order oracle query at a point from seeing the future coordinates in the later chain construction. 

\begin{lemma}\label{lem:varphi}
Let $p\in\mathbb N$, $q\in(0,1)$, and $L>0$.  There exist scalars $a>0$, $b>0$, and a function $\varphi\colon\mathbb R\to\mathbb R$ such that:
\begin{enumerate}[label=(\roman*)]
\item $\varphi\in C^{p}(\mathbb R)$,
\item $\varphi^{(k)}(0)=0$ for $0\le k\le p$,
\item $0\le\varphi'(s)\le q$ for every $s\in\mathbb R$,
\item $\mathrm{Lip}(\varphi^{(p)})\le L$,
\item $\varphi(s)=qs-b$ for $s\ge a$, and $\varphi(s)=qs+b$ for $s\le -a$,
\item $\varphi(s)\ge qs-b$ for every $s\in\mathbb R$,
\item $b\ge qa/2$.
\end{enumerate}
\end{lemma}

\begin{proof}
Fix an even $C^\infty$ function $\rho\colon\mathbb R\to[0,1]$ such that
\[
    \rho(u)=0 \quad \text{if } |u|\le \frac12,
    \qquad
    \rho(u)=1 \quad \text{if } |u|\ge 1.
\]
Set
\[
    B_p:=\max\left\{1,\operatorname{Lip}\left(\rho^{(p-1)}\right)\right\},
    \qquad
    a:=\left(\frac{qB_p}{L}\right)^{1/p},
\]
and define
\[
    \varphi(s):=qa\int_0^{s/a}\rho(u)\,du.
\]
The constant $B_p$ is finite. Indeed, $\rho$ is $C^\infty$ and constant outside
$[-1,1]$, so $\rho^{(p)}\in C_c(\mathbb R)$ and
\[
    \operatorname{Lip}\left(\rho^{(p-1)}\right)
    \le \|\rho^{(p)}\|_\infty < \infty.
\]
Since $\rho$ is even, $\varphi$ is odd. Differentiating gives
\[
    \varphi'(s)=q\rho(s/a),
\]
which proves (iii). Moreover, for $k\ge 1$,
\[
    \varphi^{(k)}(s)=qa^{1-k}\rho^{(k-1)}(s/a).
\]
Clearly $\varphi(0)=0$, and since $\rho$ vanishes near the origin, we have
$\varphi^{(k)}(0)=0$ for $1\le k\le p$. This proves (ii).

Furthermore,
\[
    \operatorname{Lip}\left(\varphi^{(p)}\right)
    \le qa^{-p}\operatorname{Lip}\left(\rho^{(p-1)}\right)
    \le qa^{-p}B_p
    = L,
\]
which proves (iv).

Next, set
\[
    b:=qa\int_0^1 (1-\rho(u))\,du.
\]
For $s\ge a$,
\[
    \varphi(s)
    = qa\int_0^1 \rho(u)\,du + q(s-a)
    = qs-b.
\]
By oddness of $\varphi$, we also have
\[
    \varphi(s)=qs+b
    \qquad \text{for } s\le -a.
\]
This proves (v). Since $\rho=0$ on $[0,1/2]$,
\[
    b\ge qa\int_0^{1/2}1\,du
    = \frac{qa}{2},
\]
which proves (vii).

It remains to verify the global lower bound (vi). For $s\ge a$, the inequality
holds with equality. For $0\le s\le a$,
\[
    \varphi(s)-qs+b
    = \int_s^a \bigl(q-\varphi'(u)\bigr)\,du
    \ge 0.
\]
For $s<0$, the bounds $0\le \varphi'\le q$ and $\varphi(0)=0$ imply
\[
    |\varphi(s)|\le q|s|,
\]
and hence
\[
    \varphi(s)\ge qs\ge qs-b.
\]
Thus the global lower bound (vi) holds.
\end{proof}

\subsection{The smooth chain map}

The next lemma constructs the smooth chain structure used in the resisting-oracle argument. Inheriting the properties of the scalar function defined in Lemma~\ref{lem:varphi}, the map is affine along the positive tail containing the fixed point, while all finite-order derivatives vanish in every undiscovered future coordinate when that coordinate is queried at zero.

\begin{lemma}\label{lem:chain-smooth}
Let $p\in\mathbb N$, $q\in(0,1)$, $L>0$, and let $\varphi,a,b$ be as in Lemma~\ref{lem:varphi}.  Let $m\in\mathbb N$, let $R\ge a$, and let $v_1,\ldots,v_m$ be an orthonormal family in $\mathbb R^n$ with $n\ge m$.  Define scalars $s_1,\ldots,s_m$ by $s_m=R$ and
\[
    s_i:=\frac{s_{i+1}+b}{q}\qquad(1\le i\le m-1),
\]
and set $c:=s_1$.  Define
\[
    F(x):=c v_1+\sum_{i=1}^{m-1}\varphi(\langle v_i,x\rangle)v_{i+1}.
\]
Then $F\in\mathcal C_{p,q,L}(\mathbb R^n)$ and $F$ has the unique fixed point
\[
    x_\star:=\sum_{i=1}^m s_i v_i.
\]
\end{lemma}

\begin{proof}
For $k\in\{1,\ldots,p\}$,
\[
    D^kF(x)[u_1,\ldots,u_k]
    =\sum_{i=1}^{m-1}\varphi^{(k)}(\langle v_i,x\rangle)
      \prod_{\ell=1}^k\langle v_i,u_\ell\rangle\,v_{i+1}.
\]
In particular,
\[
    DF(x)u=\sum_{i=1}^{m-1}\varphi'(\langle v_i,x\rangle)\langle v_i,u\rangle v_{i+1}.
\]
Since $v_2,\ldots,v_m$ are orthonormal and $|\varphi'|\le q$,
\[
    \|DF(x)u\|^2
    =\sum_{i=1}^{m-1}|\varphi'(\langle v_i,x\rangle)|^2|\langle v_i,u\rangle|^2
    \le q^2\|u\|^2.
\]
Thus $\|DF(x)\|\le q$ and $F$ is a $q$-contraction.

For unit vectors $u_1,\ldots,u_p$, write
\[
    \Delta_i:=\varphi^{(p)}(\langle v_i,x\rangle)-\varphi^{(p)}(\langle v_i,y\rangle)
\]
for some $x,y\in \mathbb{R}^n$. Then $|\Delta_i|\le L|\langle v_i,x-y\rangle|\le L\|x-y\|$, and therefore
\begin{align*}
\|(D^pF(x)-D^pF(y))[u_1,\ldots,u_p]\|^2
&=\sum_{i=1}^{m-1}|\Delta_i|^2\prod_{\ell=1}^p|\langle v_i,u_\ell\rangle|^2 \\
&\le L^2\|x-y\|^2\sum_{i=1}^{m-1}|\langle v_i,u_1\rangle|^2 \\
&\le L^2\|x-y\|^2.
\end{align*}
Taking the supremum over unit $u_1,\ldots,u_p$ gives $\mathrm{Lip}(D^pF)\le L$.

Next, since $s_i=(s_{i+1}+b)/q>s_{i+1}$, we have $s_1>\cdots>s_m=R\ge a$.  Hence
\[
    \varphi(s_i)=qs_i-b=s_{i+1}\qquad(1\le i\le m-1).
\]
Consequently,
\[
    F(x_\star)=s_1v_1+\sum_{i=1}^{m-1}\varphi(s_i)v_{i+1}
    =s_1v_1+\sum_{i=1}^{m-1}s_{i+1}v_{i+1}=x_\star.
\]
The fixed point is unique because $F$ is a contraction.
\end{proof}

\subsection{Proof of the exact fixed-point-error lower bound}\label{sec:proof-main}

We now prove the main result, Theorem~\ref{thm:error}.

\begin{proof}[Proof of Theorem~\ref{thm:error}]
Let $\varphi,a,b$ be given by Lemma~\ref{lem:varphi}.  Choose $m\ge N+1$ large enough so that
\[
    \left(\frac{1-q^{2(m-N)}}{1-q^{2m}}\right)^{1/2}\ge 1-\frac{\varepsilon}{2}.
\]
Such an $m$ exists since $q\in(0,1)$ and, for fixed $N$,
\[
\left(\frac{1-q^{2(m-N)}}{1-q^{2m}}\right)^{1/2}\to 1
\qquad
\text{as } m\to\infty .
\]
Then choose $R\ge a$ large enough so that, with
\[
    \theta:=\frac{b}{R(1-q)},
\]
we have
\[
    \frac{1}{1+\theta}\ge 1-\frac{\varepsilon}{2}.
\]
Set $n\ge n_0:=m+N$.  Define $s_m:=R$ and
\[
    s_i:=\frac{s_{i+1}+b}{q}\qquad(1\le i\le m-1),
\]
and set $c:=s_1$.

Fix the deterministic method on $\mathbb{R}^n$ and set $x_0=0$. We now use a resisting-oracle technique to choose $v_1,\dots,v_{N+1}$ inductively. Suppose that at iteration $t\in\{0,\ldots,N\}$ the points $x_0,\ldots,x_t$ are known and $v_1,\ldots,v_t$ have been chosen orthonormally.  Choose $v_{t+1}$ to be any unit vector orthogonal to
\[
    \mathrm{span}\{v_1,\ldots,v_t,x_0,\ldots,x_t\}.
\]
This is possible because $x_0=0$, so the dimension of the span is at most $2t$, and $n=m+N\ge2N+1>2t$.

After choosing $v_{t+1}$, define the truncated map
\[
    F_t(x):=c v_1+\sum_{i=1}^{t}\varphi(\langle v_i,x\rangle)v_{i+1},
\]
where the sum is empty when $t=0$.  For $t<N$, answer the method's query at $x_t$ with the derivatives up to order $p$ of $F_t$ at $x_t$.  Since the method is deterministic, this fixes $x_{t+1}$. Note that the choice at $t=N$ is used only to define $v_{N+1}$.

After $x_N$ and $v_{N+1}$ have been fixed, choose $v_{N+2},\ldots,v_m$ orthonormally in the orthogonal complement of
\[
    \mathrm{span}\{v_1,\ldots,v_{N+1},x_0,\ldots,x_N\}.
\]
This is possible because the complement has dimension at least $n-2N-1\ge m-N-1$. Now define the map
\[
    F(x):=c v_1+\sum_{i=1}^{m-1}\varphi(\langle v_i,x\rangle)v_{i+1}.
\]
By Lemma~\ref{lem:chain-smooth}, $F\in\mathcal C_{p,q,L}(\mathbb R^n)$ and
\[
    x_\star:=\sum_{i=1}^m s_i v_i
\]
is its unique fixed point.

We claim that the simulated transcript is exactly the transcript of the final map $F$ through the queries that determine $x_N$.   Fix $t\in\{0,\ldots,N-1\}$.  If $i>t$, then $v_i$ is orthogonal to $x_t$: either $v_i$ was chosen later against a span containing $x_t$, or it was chosen in the final tail subspace orthogonal to all $x_0,\ldots,x_N$.  Hence $\langle v_i,x_t\rangle=0$.  Since $\varphi^{(k)}(0)=0$ for $0\le k\le p$, every term with $i>t$ contributes zero to $D^kF(x_t)$ for $0\le k\le p$.  Therefore
\[
    D^kF(x_t)=D^kF_t(x_t)\qquad(0\le k\le p,
    \;0\le t<N).
\]
Thus $x_1,\ldots,x_N$ are exactly the iterates produced by the method when run on the final map $F$.

Now by construction,
\[
    x_N\perp v_i\qquad\text{for all } i\ge N+1.
\]
Therefore
\[
    \|x_N-x_\star\|^2\ge \sum_{i=N+1}^{m}s_i^2.
\]
Define
\[
    g_i:=Rq^{-(m-i)}\qquad(1\le i\le m).
\]
We compare $s_i$ with $g_i$.  Since $s_m=g_m=R$, iterating the backward recursion gives
\[
    s_i=Rq^{-(m-i)}+b\sum_{r=1}^{m-i}q^{-r}
    =g_i+b\sum_{r=1}^{m-i}q^{-r}.
\]
Thus $s_i\ge g_i$.  On the other hand, we have
\[
    \sum_{r=1}^{m-i}q^{-r}=
    q^{-(m-i)}(1+q+\cdots+q^{m-i-1})\le\frac{q^{-(m-i)}}{1-q}.
\]
Therefore
\[
    s_i\le g_i\left(1+\frac{b}{R(1-q)}\right)=(1+\theta)g_i
    \qquad(1\le i\le m).
\]
Using this comparison,
\[
    \frac{\|x_N-x_\star\|^2}{\|x_\star\|^2}
    \ge
    \frac{\sum_{i=N+1}^{m}s_i^2}{\sum_{i=1}^{m}s_i^2}
    \ge
    \frac{\sum_{i=N+1}^{m}g_i^2}{(1+\theta)^2\sum_{i=1}^{m}g_i^2}.
\]
Taking square roots gives
\[
    \frac{\|x_N-x_\star\|}{\|x_\star\|}
    \ge
    \frac{1}{1+\theta}
    \left(\frac{\sum_{i=N+1}^{m}g_i^2}{\sum_{i=1}^{m}g_i^2}\right)^{1/2}.
\]
Since $g_i^2=R^2q^{-2(m-i)}$, the ratio of geometric sums is
\[
    \frac{\sum_{i=N+1}^{m}g_i^2}{\sum_{i=1}^{m}g_i^2}
    =\frac{\sum_{i=N+1}^{m}q^{2i}}{\sum_{i=1}^{m}q^{2i}}
    =q^{2N}\frac{1-q^{2(m-N)}}{1-q^{2m}}.
\]
Consequently,
\[
    \frac{\|x_N-x_\star\|}{\|x_\star\|}
    \ge
    \frac{1}{1+\theta}q^N
    \left(\frac{1-q^{2(m-N)}}{1-q^{2m}}\right)^{1/2}
    \ge
    \left(1-\frac{\varepsilon}{2}\right)^2q^N
    \ge (1-\varepsilon)q^N.
\]
Finally, since $x_0=0$, we have $\|x_\star\|=\|x_0-x_\star\|$, and hence
\[
    \|x_N-x_\star\|\ge (1-\varepsilon)q^N\|x_0-x_\star\|.
\]
For an arbitrary starting point $x_0\in\mathbb{R}^n$, apply the above
construction in the translated coordinates $z:=x-x_0$: if $F$ is
the map constructed above, define
\[
    F_{x_0}(x):=x_0+F(x-x_0).
\]
Then $F_{x_0}\in\mathcal C_{p,q,L}(\mathbb{R}^n)$, and its unique fixed point is
\[
    x_\star^{x_0}:=x_0+x_\star.
\]
The derivatives of $F_{x_0}$ are just the translated derivatives of
$F$, so the resisting-oracle argument applies verbatim.
\end{proof}

\section{Residual lower bound for contractive and nonexpansive maps}\label{sec:residual}
In the practical use of fixed-point iterations, the fixed-point residual
$\|x-Tx\|$ is computable while the distance to the solution $\|x-x_\star\|$ is
not, so the residual serves as the practical termination criterion. For
nonexpansive operators (the case $q=1$), prior work established matching
$\Theta(1/N)$ complexity bounds on the residual in the value-oracle model: the
rate $\frac{2}{N+1}\|x_0-x_\star\|$ is attained by the Halpern iteration
\cite{lieder2021convergence}, and matching lower bounds were given in
\cite{contreras2023optimal,park2022exact}. For $q$-contractive operators with
$q\in(0,1)$, the exact optimal value-oracle rate is
$(1+q)\bigl(\sum_{j=0}^{N}q^{-j}\bigr)^{-1}\|x_0-x_\star\|$, with the upper
bound attained by the method OC-Halpern and a matching lower bound established in
\cite{park2022exact}.

In this section, we show that these exact residual lower bounds continue to
hold for any fixed finite oracle order, using a construction similar to that
of Section~\ref{sec:chain}. The exact residual lower bound is formalized as follows:

\begin{theorem}\label{thm:residual}
Let $p,N\in\mathbb N$, $q\in(0,1)$, and $L>0$.  Set $n\ge 2N+1$. For every $x_0\in \mathbb{R}^n$, for every deterministic $p$-th order method initialized at $x_0$, and for $\varepsilon \in (0,1)$, there exists  a map $F\in\mathcal C_{p,q,L}(\mathbb R^n)$ with a unique fixed point $x_\star$ such that the method's $N$-th iterate $x_N$ satisfies
\[
    \|x_N-F(x_N)\|\ge (1-\varepsilon)(1+q)\left(\sum_{j=0}^{N}q^{-j}\right)^{-1}\|x_0-x_\star\|.
\]
\end{theorem}

Note the following residual guarantee of the Picard iteration.
\begin{proposition}[Picard residual bound \cite{ortega2000iterative}]\label{prop:picard-residual}
Let $G\colon\mathbb R^n\to\mathbb R^n$ be $C^1$, let $x_\star$ be the fixed point of $G$, and assume $\|DG(x)\|\le q<1$ for all $x$.  Then Picard iteration $x_{t+1}=G(x_t)$ satisfies
\[
    \|x_N-G(x_N)\|\le (1+q)q^N\|x_0-x_\star\|.
\]
This bound is tight in the sense that there exists an operator $G^\star$ such that 
\[
    \|x_N-G^\star(x_N)\| = (1+q)q^N\|x_0-x_\star\|.
\]
\end{proposition}

Comparing the lower bound of Theorem~\ref{thm:residual} with the upper bound
of Proposition~\ref{prop:picard-residual}, we have
\[
(1-q)\!\!\!\!
\underbrace{(1+q)q^N}_{\substack{\text{upper bound of}\\ \text{Proposition~\ref{prop:picard-residual}}}}
\le
\underbrace{(1+q)\left(\sum_{j=0}^{N}q^{-j}\right)^{-1}}_{\substack{\text{lower bound of}\\ \text{ Theorem~\ref{thm:residual} as }\varepsilon\to 0}}
\le
\underbrace{(1+q)q^N}_{\substack{\text{upper bound of}\\ \text{Proposition~\ref{prop:picard-residual}}}}\!\!\!\!.
\]
Thus, for each fixed $q\in(0,1)$, the worst-case residual rate $O(q^N)$ of
Picard iteration is optimal up to the factor $1-q$.

But in fact, the lower bound
of Theorem~\ref{thm:residual} is tight, as it exactly matches the residual bound
of OC-Halpern~\cite{park2022exact}:

\begin{proposition}[OC-Halpern residual bound \cite{park2022exact}]\label{prop:halpern}
Let $q\in(0,1]$, and let $F\colon\mathbb R^n\to\mathbb R^n$ be a $q$-contraction with a fixed point $x_\star$. Then the iterates $\{x_N\}_{N\ge0}$ of OC-Halpern satisfy
\[
    \|x_N-F(x_N)\| \le (1+q)\left(\sum_{j=0}^{N}q^{-j}\right)^{-1}\|x_0-x_\star\|.
\]
\end{proposition}
\begin{corollary}\label{cor:minimaxerrorresidual}
Let $q\in(0,1]$. No deterministic method can improve upon OC-Halpern's factor
$(1+q)\bigl(\sum_{j=0}^{N}q^{-j}\bigr)^{-1}$ in worst-case residual after $N$
oracle queries, even when it is allowed to query all derivatives of $F$ up to
any fixed finite order $p$. In particular, in the nonexpansive case $q=1$,
the factor equals $\frac{2}{N+1}$, so the $O(1/N)$ residual rate of the
Halpern iteration~\cite{lieder2021convergence} likewise cannot be improved by
finite-order methods.
\end{corollary}

The rest of this section is devoted to the proof of Theorem~\ref{thm:residual}. We first prove some lemmas for the residual lower bound setting, and then combine this result with the resisting-oracle choice of the directions $v_1,\ldots,v_{N+1}$ as in the proof of Theorem~\ref{thm:error}.

\subsection{Residual lower bound on the hidden hyperplane}
 First, we construct a smooth chain structure as in Lemma~\ref{lem:chain-smooth}. The idea is similar to that of the previous section, but the mapping is defined slightly differently.

\begin{lemma}\label{lem:cyclic-chain-smooth}
Let $p\in\mathbb N$, $q\in(0,1)$, $L>0$, and let $\varphi,a,b$ be as in
Lemma~\ref{lem:varphi}.  Let $m\in\mathbb N$, let $R\ge a$, and let
$v_1,\ldots,v_m$ be an orthonormal family in $\mathbb R^n$ with $n\ge m$.
Define scalars $s_1,\ldots,s_m$ by $s_m=R$ and
\[
    s_i:=\frac{s_{i+1}+b}{q}\qquad(1\le i\le m-1).
\]
Set
\[
    c:=s_1+qR-b.
\]
Define
\[
    F(x):=
    c v_1 -\varphi(\langle v_m,x\rangle)v_1+\sum_{i=1}^{m-1}\varphi(\langle v_i,x\rangle)v_{i+1}.
\]
Then $F\in\mathcal C_{p,q,L}(\mathbb R^n)$ and $F$ has the unique fixed point
\[
    x_\star:=\sum_{i=1}^m s_i v_i.
\]
\end{lemma}

\begin{proof}
For $k\in\{1,\ldots,p\}$,
\[
\begin{aligned}
    D^kF(x)[u_1,\ldots,u_k]
    &= -\varphi^{(k)}(\langle v_m,x\rangle)
      \prod_{\ell=1}^k\langle v_m,u_\ell\rangle\,v_1  \\
    &\quad+\sum_{i=1}^{m-1}
    \varphi^{(k)}(\langle v_i,x\rangle)\prod_{\ell=1}^k\langle v_i,u_\ell\rangle\,v_{i+1}.
\end{aligned}
\]
In particular,
\[
    DF(x)u
    =
    -\varphi'(\langle v_m,x\rangle)\langle v_m,u\rangle v_1
    +
    \sum_{i=1}^{m-1}
    \varphi'(\langle v_i,x\rangle)\langle v_i,u\rangle v_{i+1}.
\]
Since $v_1,\ldots,v_m$ are orthonormal and $|\varphi'|\le q$,
\[
\begin{aligned}
    \|DF(x)u\|^2
    &=
    |\varphi'(\langle v_m,x\rangle)|^2|\langle v_m,u\rangle|^2 + \sum_{i=1}^{m-1}
    |\varphi'(\langle v_i,x\rangle)|^2|\langle v_i,u\rangle|^2 \\
    &\le q^2\sum_{i=1}^{m}|\langle v_i,u\rangle|^2
    \le q^2\|u\|^2.
\end{aligned}
\]
Thus $\|DF(x)\|\le q$ for every $x$, so $F$ is a $q$-contraction. It remains to check the $p$-th derivative Lipschitz bound. For unit vectors
$u_1,\ldots,u_p$, write
\[
    \Delta_i:=\varphi^{(p)}(\langle v_i,x\rangle)-
    \varphi^{(p)}(\langle v_i,y\rangle)\qquad(1\le i\le m)
\]
for some $x,y\in \mathbb{R}^n$.
Then
\[
    |\Delta_i| \le L|\langle v_i,x-y\rangle|.
\]
Therefore
\[
\begin{aligned}
    &\|(D^pF(x)-D^pF(y))[u_1,\ldots,u_p]\|^2 \\
    &\quad = |\Delta_m|^2\prod_{\ell=1}^p|\langle v_m,u_\ell\rangle|^2+\sum_{i=1}^{m-1}|\Delta_i|^2\prod_{\ell=1}^p|\langle v_i,u_\ell\rangle|^2  \\
    &\quad \le L^2  \sum_{i=1}^{m} |\langle v_i,x-y\rangle|^2 \prod_{\ell=1}^p|\langle v_i,u_\ell\rangle|^2  \\
    &\quad \le  L^2\|x-y\|^2  \sum_{i=1}^{m}|\langle v_i,u_1\rangle|^2 \le L^2\|x-y\|^2.
\end{aligned}
\]
Taking the supremum over unit $u_1,\ldots,u_p$ gives
\[
    \mathrm{Lip}(D^pF)\le L.
\]
Finally, since $s_m=R\ge a$ and
\[
    s_i=\frac{s_{i+1}+b}{q}>s_{i+1} \qquad(1\le i\le m-1),
\]
we have $s_1>\cdots>s_m=R\ge a$.  Hence
\[
    \varphi(s_i)=qs_i-b=s_{i+1} \qquad(1\le i\le m-1),
\]
and
\[
    \varphi(s_m)=qR-b.
\]
By the definition of $c$,
\[
    c-\varphi(s_m)  = s_1+qR-b-(qR-b) =  s_1.
\]
Therefore
\[
\begin{aligned}
    F(x_\star)
    &=
    \bigl(c-\varphi(s_m)\bigr)v_1
    +
    \sum_{i=1}^{m-1}\varphi(s_i)v_{i+1} \\
    &=
    s_1v_1+\sum_{i=1}^{m-1}s_{i+1}v_{i+1}
    =
    x_\star.
\end{aligned}
\]
The fixed point is unique because $F$ is a strict contraction.
\end{proof}

\paragraph{Choice of constants.}
 For the rest of this section, fix $p,N\in\mathbb N$, $q\in(0,1)$, $L>0$, and $\varepsilon\in(0,1)$. We now choose $R\ge a$ in Lemma~\ref{lem:cyclic-chain-smooth}, and then  establish several auxiliary lemmas.
 Define
\[
    S_1:=\sum_{j=0}^{N-1}q^j,
    \qquad
    S_2:=\sum_{j=0}^{N}q^{2j}.
\]
Choose $R\ge a$ sufficiently large so that
\begin{equation}\label{eq:Rchoice1}
     R(1-q)>a+b
\end{equation}
and, with
\[
    \theta:=\frac{b}{R(1-q)},
    \qquad
    \delta:=\frac{b\sqrt{S_2}}{R(1+q^{N+1})},
\]
we have
\begin{equation}\label{eq:Rchoice2}
    \theta+\delta\le\varepsilon.
\end{equation}
Let $s_{N+1}=R$ and
\[
    s_i:=\frac{s_{i+1}+b}{q}\qquad(1\le i\le N),
\]
so that
\[
    s_1=q^{-N}(R+bS_1), \qquad c:=s_1+qR-b.
\]
Define
\[
   F(x):=c v_1-\varphi(\langle v_{N+1},x\rangle)v_1+\sum_{i=1}^{N}\varphi(\langle v_i,x\rangle)v_{i+1}.
\]
By Lemma~\ref{lem:cyclic-chain-smooth}, this map belongs to $\mathcal C_{p,q,L}$ and has fixed point $x_\star=\sum_{i=1}^{N+1} s_i v_i$. We first estimate an upper bound on the distance to the origin from the fixed point $x_\star$. 

\begin{lemma}\label{lem:residual-xstar-bound}
With the above definitions,
\[
    \|x_\star\|\le(1+\theta)q^{-N}R\sqrt{S_2},
\]
\end{lemma}

\begin{proof}
Define
\[
    g_i:=Rq^{-(N+1-i)}
    \qquad(1\le i\le N+1).
\]
The recursion gives
\[
    s_i = g_i+b\sum_{r=1}^{N+1-i}q^{-r}.
\]
Thus $s_i\ge g_i$. Moreover,
\[
    \sum_{r=1}^{N+1-i}q^{-r}\le \frac{q^{-(N+1-i)}}{1-q},
\]
so
\[
    s_i \le g_i\left(1+\frac{b}{R(1-q)}\right)
    = (1+\theta)g_i.
\]
Consequently,
\[
    \|x_\star\|
    \le
    (1+\theta)
    \left(\sum_{i=1}^{N+1}g_i^2\right)^{1/2}
    =
    (1+\theta)q^{-N}R\sqrt{S_2}.
\]
\end{proof}

We next analyze an affine proxy for the chain map. On the positive tail, $\varphi$ agrees with the affine map $s\mapsto qs-b$. The following lemma shows that the missing final coordinate forces residual at least $\frac{1+q^{N+1}}{\sqrt{S_2}}R$ on the
hidden hyperplane.

\begin{lemma}\label{lem:affine-lowerbound}
Let
\[
    H:=\{x\in\mathbb R^n:\langle v_{N+1},x\rangle=0\}
\]
and define
\[
    F_{\rm aff}(x):= cv_1 -\bigl(q\langle v_{N+1},x\rangle-b\bigr)v_1+\sum_{i=1}^{N}\bigl(q\langle v_i,x\rangle-b\bigr)v_{i+1}.
\]
Then
\[
    \min_{x\in H}\|x-F_{\rm aff}(x)\| =\frac{1+q^{N+1}}{\sqrt{S_2}}R.
\]
\end{lemma}

\begin{proof}
Write $z=x-x_\star$ and $r=x-F_{\rm aff}(x)$.  In the coordinates
$v_1,\ldots,v_{N+1}$, we have
\begin{align*}
r_1&=x_1-(c-qx_{N+1}+b) \\
&=x_1-\bigl(s_1+qs_{N+1}-b\bigr)+qx_{N+1}-b \\
&=(x_1-s_1)+q(x_{N+1}-s_{N+1}) \\
&=z_1+qz_{N+1}.
\end{align*}
and for $1\le i \le N$,
\begin{align*}
r_{i+1}&=x_{i+1}-(qx_i-b)\\
&=(x_{i+1}-s_{i+1})-q(x_i-s_i) \\
&=z_{i+1}-qz_i.
\end{align*}
Since $x\in H$, we have $x_{N+1}=0$, and hence
\[
z_{N+1}=x_{N+1}-s_{N+1}=-R.
\]
Using the coordinate identities above, for the projection $P\colon r\rightarrow \mathrm{span}\{v_1,\ldots,v_{N+1}\}$, we get:
\begin{align*}
Pr&=(z_1-qR)v_1+\sum_{i=1}^{N-1}(z_{i+1}-qz_i)v_{i+1}
+(-R-qz_N)v_{N+1} \\
&=-qR v_1-Rv_{N+1}+\sum_{i=1}^{N}z_i(v_i-qv_{i+1})\\
&\in -qR v_1-Rv_{N+1}+\mathrm{span}\{ v_1-qv_2,\, v_2-qv_3,\,\ldots,\, v_N-qv_{N+1}\}
\end{align*}
The orthogonal complement of $\mathrm{span}\{ v_1-qv_2,\, v_2-qv_3,\,\ldots,\, v_N-qv_{N+1}\}$ inside $\mathrm{span}\{v_1,\ldots,v_{N+1}\}$
is generated by
\[
    w:=q^Nv_1+q^{N-1}v_2+\cdots+qv_N+v_{N+1}.
\]
with $\|w\|=\sqrt{S_2}$. Therefore, 
\[
   \|r\|\ge  \|Pr\|\ge\frac{|\langle w,Pr\rangle|}{\|w\|}=\frac{R(1+q^{N+1})}{\sqrt{S_2}}.
\]
Equality is attained by choosing the free variables $z_1,\ldots,z_N$ so that the projected residual $Pr$ is parallel to $w$.
 Hence
\[
    \min_{x\in H}\|x-F_{\rm aff}(x)\| = \frac{1+q^{N+1}}{\sqrt{S_2}}R.
\]
\end{proof}

The actual map $F$ is not globally affine, because $\varphi$ is flattened near the origin to hide finite-order derivative information. Nevertheless, a point
on the hidden hyperplane with residual smaller than $\frac{1+q^{N+1}}{\sqrt{S_2}}R$ would be forced into the affine region of every revealed coordinate. The following lemma makes this argument rigorous.

\begin{lemma}\label{lem:hyperplane}
For every $x\in H$,
\[
    \|x-F(x)\|\ge\frac{1+q^{N+1}}{\sqrt{S_2}}R-b>0.
\]
\end{lemma}
\begin{proof}
First, note that 
\[
    S_2=\sum_{j=0}^{N}q^{2j}\le \frac{1}{1-q^2},
\]
and hence
\[
    \frac{1+q^{N+1}}{\sqrt{S_2}}
    \ge \sqrt{1-q^2}
    \ge 1-q.
\]
Therefore, by the choice of $R$ in \eqref{eq:Rchoice1},
\[
    \frac{1+q^{N+1}}{\sqrt{S_2}}R-b
    \ge R(1-q)-b>a>0.
\]
Now, fix $x\in H$ and write
\[
    x=\sum_{i=1}^{N+1}x_i v_i+x_\perp,
    \qquad x_\perp\perp\mathrm{span}\{v_1,\ldots,v_{N+1}\}.
\]
Since $x\in H$, we have $x_{N+1}=0$.  Let
\[
    r:=x-F(x)=\sum_{i=1}^{N+1}r_i v_i+x_\perp.
\]
Then
\[
    r_1=x_1-c,
    \qquad
    r_{i+1}=x_{i+1}-\varphi(x_i)\qquad(1\le i\le N).
\]
Using $c=s_1+qR-b$, we get
\[
    x_1=s_1+r_1+qR-b.
\]
Moreover, since $\varphi(s)\ge qs-b$ for all $s$,
\[
    r_{i+1}=x_{i+1}-\varphi(x_i)\le x_{i+1}-qx_i+b,
\]
and hence
\[
x_{i+1}\ge qx_i-b+r_{i+1} \qquad(1\le i\le N).
\]
Using $s_{i+1}=qs_i-b$, induction gives
\[
x_i\ge s_i+q^iR-q^{i-1}b+\sum_{k=1}^{i}q^{i-k}r_k
\qquad(1\le i\le N).
\]
Suppose, for the sake of contradiction, that
\[
\|r\|<\frac{1+q^{N+1}}{\sqrt{S_2}}R-b,
\]
Then
\[
\left|\sum_{k=1}^{i}q^{i-k}r_k\right| \le \sqrt{S_2}\|r\|
< R(1+q^{N+1}) \qquad(1\le i\le N).
\]
Also we showed $s_i\ge Rq^{-(N+1-i)}$ in the proof of Lemma~\ref{lem:residual-xstar-bound}. Therefore
\begin{align*}
x_i &> Rq^{-(N+1-i)}+q^iR-q^{i-1}b-R(1+q^{N+1}) \\
&= R\bigl(q^{-(N+1-i)}-1+q^i-q^{N+1}\bigr)-q^{i-1}b \\
&\ge R\bigl(q^{-1}-1\bigr)-b \\
&= R\frac{1-q}{q}-b \\
&\ge R(1-q)-b\\
&>a.
\end{align*}
Thus
\[
\varphi(x_i)=qx_i-b \qquad(1\le i\le N).
\]
Since $x\in H$, we also have $x_{N+1}=0$, so $\varphi(x_{N+1})=0$.
Hence, 
\[
F(x)=F_{\rm aff}(x)-b v_1.
\]
Therefore
\[
x-F_{\rm aff}(x)=r-bv_1.
\]
By Lemma~\ref{lem:affine-lowerbound},
\[
\frac{1+q^{N+1}}{\sqrt{S_2}}R \le \|x-F_{\rm aff}(x)\| =\|r-bv_1\| \le \|r\|+b.
\]
Thus
\[
\|r\|\ge \frac{1+q^{N+1}}{\sqrt{S_2}}R-b,
\]
contradicting our assumption. Therefore, 
\[
\|x-F(x)\| \ge \frac{1+q^{N+1}}{\sqrt{S_2}}R-b.
\]
\end{proof}

\subsection{Proof of the complexity lower bound}
We now transfer the lower bound on the hyperplane into a complexity lower bound. Similar to the proof of Theorem~\ref{thm:error}, the directions are chosen adaptively, after seeing the method's previous queries, so that each new direction is orthogonal to all information generated so far.

\begin{proof}[Proof of Theorem~\ref{thm:residual}]
Fix the deterministic method on $\mathbb R^n$, where $n\ge 2N+1$, and
without loss of generality, set $x_0:=0$. We construct $v_1,\ldots,v_{N+1}$ inductively.  Suppose that at iteration
$t\in\{0,\ldots,N\}$, the points $x_0,\ldots,x_t$ are known and
$v_1,\ldots,v_t$ have been chosen orthonormally.  Choose $v_{t+1}$ to be any
unit vector orthogonal to
\[
   \mathrm{span}\{v_1,\ldots,v_t,x_0,\ldots,x_t\}.
\]
This is possible because $x_0=0$, so the dimension of the span is at most
$2t$, while $n\ge 2N+1>2t$.

After choosing $v_{t+1}$, define the truncated map
\[
    F_t(x):=
    c v_1+
    \sum_{i=1}^{t}\varphi(\langle v_i,x\rangle)v_{i+1},
\]
where the sum is empty when $t=0$.  For $t<N$, answer the method's query at
$x_t$ with the derivatives up to order $p$ of $F_t$ at $x_t$.  Since the method
is deterministic, this fixes $x_{t+1}$.  The choice at $t=N$ is used only to
define the final hidden direction $v_{N+1}$.
Now define the final chain map
\[
    F(x):= c v_1 -\varphi(\langle v_{N+1},x\rangle)v_1  +\sum_{i=1}^{N}\varphi(\langle v_i,x\rangle)v_{i+1}.
\]
By Lemma~\ref{lem:cyclic-chain-smooth}, $F\in\mathcal C_{p,q,L}(\mathbb R^n)$
and has the unique fixed point
\[
    x_\star:=\sum_{i=1}^{N+1}s_i v_i.
\]
Moreover, by Lemma~\ref{lem:residual-xstar-bound},
\[
    \|x_\star\|
    \le
    (1+\theta)q^{-N}R\sqrt{S_2}.
\]

We claim that the simulated transcript is exactly the transcript of the final
map $F$ through the queries that determine $x_N$.  Fix
$t\in\{0,\ldots,N-1\}$.  If $i>t$, then $v_i$ is orthogonal to $x_t$.  Hence
\[
    \langle v_i,x_t\rangle=0 \qquad(i>t).
\]
Since
\[
    \varphi^{(k)}(0)=0  \qquad(0\le k\le p),
\]
every term with $i>t$ contributes zero to $D^kF(x_t)$ for $0\le k\le p$.
In particular, the term
\[
    -\varphi(\langle v_{N+1},x\rangle)v_1
\]
also contributes zero at every query $x_t$ with $t<N$.  Therefore
\[
    D^kF(x_t)=D^kF_t(x_t)  \qquad  (0\le k\le p,\ 0\le t<N).
\]
Thus $x_1,\ldots,x_N$ are exactly the iterates produced by the method when run on the final map $F$.

The final vector $v_{N+1}$ was chosen orthogonal to $x_N$, so
\[
    x_N\in H:=\{x\in\mathbb R^n:\langle v_{N+1},x\rangle=0\}.
\]
By Lemma~\ref{lem:hyperplane},
\[
    \|x_N-F(x_N)\|\ge \frac{1+q^{N+1}}{\sqrt{S_2}}R-b.
\]
Since $x_0=0$, we also have
\[
    \|x_0-x_\star\| = \|x_\star\| \le (1+\theta)q^{-N}R\sqrt{S_2}.
\]
Combining the last two inequalities gives
\begin{align*}
    \frac{\|x_N-F(x_N)\|}{\|x_0-x_\star\|}
    &\ge  \frac{\frac{1+q^{N+1}}{\sqrt{S_2}}R-b}{(1+\theta)q^{-N}R\sqrt{S_2}} \\
    &=\frac{1}{1+\theta}
    \left(1-\frac{b\sqrt{S_2}}{R(1+q^{N+1})}\right)
    \frac{q^N(1+q^{N+1})}{S_2}.
\end{align*}
By the definitions of $\theta,\delta$ and the choice of $R$,
\[
    \frac{1}{1+\theta}\left(1-\frac{b\sqrt{S_2}}{R(1+q^{N+1})} \right)=\frac{1}{1+\theta}(1-\delta)\ge 1-\theta-\delta\ge 1-\varepsilon.
\]
Since
\[
    S_2=\sum_{j=0}^{N}q^{2j} = \left(\sum_{j=0}^{N}q^j\right)\frac{1+q^{N+1}}{1+q},
\]
we have
\[
    \frac{1+q^{N+1}}{S_2} =  \frac{1+q}{\sum_{j=0}^{N}q^j}.
\]
Therefore
\begin{align*}
    \frac{\|x_N-F(x_N)\|}{\|x_0-x_\star\|}
    &\ge(1-\varepsilon)(1+q)\frac{q^N}{\sum_{j=0}^{N}q^j}  \\
    &=(1-\varepsilon)(1+q)  \left(\sum_{j=0}^{N}q^{-j}\right)^{-1}.
\end{align*}
Hence
\[
    \|x_N-F(x_N)\| \ge  (1-\varepsilon)(1+q) \left(\sum_{j=0}^{N}q^{-j}\right)^{-1} \|x_0-x_\star\|.
\]
For an arbitrary starting point $x_0\in\mathbb R^n$, the result follows by applying the above construction in the translated coordinates $z:=x-x_0$, exactly as in the proof of Theorem~\ref{thm:error}. We omit the details. 
\end{proof}

\section{Conclusion}\label{sec:conclusion}
In this work, we showed that finite-order derivative oracles do not accelerate
the computation of fixed points: for smooth $q$-contractions, no deterministic
$p$-th order method can improve upon the exact worst-case rates of Picard
iteration for the fixed-point error and of OC-Halpern for the fixed-point
residual.

As an avenue for future work, we note that our constructions require the
ambient dimension $n$ to be large relative to the query budget $N$ (e.g.,
$n\ge 2N+1$), as is standard for dimension-free lower bounds. In convex
optimization, dimension-dependent methods can be substantially faster than
dimension-free methods when the dimension is small: classical localization
schemes such as the center-of-gravity method achieve an
$O(n\log(1/\varepsilon))$ oracle complexity~\cite{ayu1965algorithm,newman1965location},
with a matching $\Omega(n\log(1/\varepsilon))$ lower
bound~\cite{nemirovsky1983problem}. Developing an analogous
dimension-dependent oracle-complexity theory for fixed-point problems with
higher-order oracles would be an interesting direction.

\section*{Acknowledgments}
UJ and EKR were supported by the Air Force Office of Scientific Research under award number FA95502510183.

\bibliographystyle{plain}
\bibliography{ref}

\begin{thebibliography}{10}

\bibitem{ayu1965algorithm}
Levin A~Yu.
\newblock On an algorithm for the minimization of convex functions.
\newblock {\em Soviet Mathematics Doklady}, 6:286--290, 1965.

\bibitem{arjevani2017oracle}
Yossi Arjevani and Ohad Shamir.
\newblock Oracle complexity of second-order methods for finite-sum problems.
\newblock {\em International Conference on Machine Learning}, 2017.

\bibitem{arjevani2019oracle}
Yossi Arjevani, Ohad Shamir, and Ron Shiff.
\newblock Oracle complexity of second-order methods for smooth convex optimization.
\newblock {\em Mathematical Programming}, 178(1):327--360, 2019.

\bibitem{banach1922operations}
Stefan Banach.
\newblock Sur les op{\'e}rations dans les ensembles abstraits et leur application aux {\'e}quations int{\'e}grales.
\newblock {\em Fundamenta mathematicae}, 3(1):133--181, 1922.

\bibitem{bauschke2017convex}
Heinz~H Bauschke and Patrick~L Combettes.
\newblock {\em Convex Analysis and Monotone Operator Theory in Hilbert Spaces}.
\newblock Springer, 2017.

\bibitem{carmon2017convex}
Yair Carmon, John~C Duchi, Oliver Hinder, and Aaron Sidford.
\newblock “{C}onvex until proven guilty”: Dimension-free acceleration of gradient descent on non-convex functions.
\newblock {\em International Conference on Machine Learning}, 2017.

\bibitem{carmon2018accelerated}
Yair Carmon, John~C Duchi, Oliver Hinder, and Aaron Sidford.
\newblock Accelerated methods for nonconvex optimization.
\newblock {\em SIAM Journal on Optimization}, 28(2):1751--1772, 2018.

\bibitem{carmon2020lower}
Yair Carmon, John~C Duchi, Oliver Hinder, and Aaron Sidford.
\newblock Lower bounds for finding stationary points {I}.
\newblock {\em Mathematical Programming}, 184(1):71--120, 2020.

\bibitem{carmon2021lower}
Yair Carmon, John~C Duchi, Oliver Hinder, and Aaron Sidford.
\newblock Lower bounds for finding stationary points {II}: First-order methods.
\newblock {\em Mathematical Programming}, 185(1):315--355, 2021.

\bibitem{cartis2011adaptive}
Coralia Cartis, Nicholas~IM Gould, and Philippe~L Toint.
\newblock Adaptive cubic regularisation methods for unconstrained optimization. part {I}: motivation, convergence and numerical results.
\newblock {\em Mathematical Programming}, 127(2):245--295, 2011.

\bibitem{chari2026optimal}
Govind~M Chari, Uijeong Jang, Ernest~K Ryu, and Beh{\c{c}}et A{\c{c}}{\i}kme{\c{s}}e.
\newblock Optimal acceleration for proximal minimization of the sum of convex and strongly convex functions.
\newblock {\em arXiv preprint arXiv:2605.08593}, 2026.

\bibitem{contreras2023optimal}
Juan~Pablo Contreras and Roberto Cominetti.
\newblock Optimal error bounds for non-expansive fixed-point iterations in normed spaces.
\newblock {\em Mathematical Programming}, 199(1):343--374, 2023.

\bibitem{diakonikolas2020halpern}
Jelena Diakonikolas.
\newblock Halpern iteration for near-optimal and parameter-free monotone inclusion and strong solutions to variational inequalities.
\newblock {\em Conference on Learning Theory}, 2020.

\bibitem{drori2017exact}
Yoel Drori.
\newblock The exact information-based complexity of smooth convex minimization.
\newblock {\em Journal of Complexity}, 39:1--16, 2017.

\bibitem{drori2022oracle}
Yoel Drori and Adrien Taylor.
\newblock On the oracle complexity of smooth strongly convex minimization.
\newblock {\em Journal of Complexity}, 68:101590, 2022.

\bibitem{jang2025computer}
Uijeong Jang, Shuvomoy~Das Gupta, and Ernest~K Ryu.
\newblock Computer-assisted design of accelerated composite optimization methods: {OptISTA}.
\newblock {\em Mathematical Programming}, pages 1--109, 2025.

\bibitem{jin2018accelerated}
Chi Jin, Praneeth Netrapalli, and Michael~I Jordan.
\newblock Accelerated gradient descent escapes saddle points faster than gradient descent.
\newblock {\em Conference on Learning Theory}, 2018.

\bibitem{kelley1995iterative}
Carl~T Kelley.
\newblock {\em Iterative Methods for Linear and Nonlinear Equations}.
\newblock SIAM, 1995.

\bibitem{kim2021accelerated}
Donghwan Kim.
\newblock Accelerated proximal point method for maximally monotone operators.
\newblock {\em Mathematical Programming}, 190(1):57--87, 2021.

\bibitem{kovalev2022first}
Dmitry Kovalev and Alexander Gasnikov.
\newblock The first optimal acceleration of high-order methods in smooth convex optimization.
\newblock {\em Neural Information Processing Systems}, 2022.

\bibitem{li2023restarted}
Huan Li and Zhouchen Lin.
\newblock Restarted nonconvex accelerated gradient descent: No more polylogarithmic factor in the in the o (epsilon\^{}(-7/4)) complexity.
\newblock {\em Journal of Machine Learning Research}, 24(157):1--37, 2023.

\bibitem{lieder2021convergence}
Felix Lieder.
\newblock On the convergence rate of the halpern-iteration.
\newblock {\em Optimization letters}, 15(2):405--418, 2021.

\bibitem{nemirovsky1983problem}
Arkadi Nemirovsky and David~B Yudin.
\newblock {\em Problem Complexity and Method Efficiency in Optimization}.
\newblock Wiley-Interscience, 1983.

\bibitem{nesterov2018lectures}
Yurii Nesterov.
\newblock {\em Lectures on Convex Optimization}, volume 137.
\newblock Springer, 2018.

\bibitem{nesterov2006cubic}
Yurii Nesterov and Boris~T Polyak.
\newblock Cubic regularization of {N}ewton method and its global performance.
\newblock {\em Mathematical Programming}, 108(1):177--205, 2006.

\bibitem{newman1965location}
Donald~J Newman.
\newblock Location of the maximum on unimodal surfaces.
\newblock {\em Journal of the ACM (JACM)}, 12(3):395--398, 1965.

\bibitem{ortega2000iterative}
James~M Ortega and Werner~C Rheinboldt.
\newblock {\em Iterative solution of nonlinear equations in several variables}.
\newblock SIAM, 2000.

\bibitem{park2022exact}
Jisun Park and Ernest~K Ryu.
\newblock Exact optimal accelerated complexity for fixed-point iterations.
\newblock {\em International Conference on Machine Learning}, 2022.

\bibitem{picard1890memoire}
{\'E}mile Picard.
\newblock Memoire sur la theorie des equations aux derivees partielles et la methode des approximations successives.
\newblock {\em Journal de Math{\'e}matiques pures et appliqu{\'e}es}, 6:145--210, 1890.

\bibitem{sabach2017first}
Shoham Sabach and Shimrit Shtern.
\newblock A first order method for solving convex bilevel optimization problems.
\newblock {\em SIAM Journal on Optimization}, 27(2):640--660, 2017.

\bibitem{zhou2026sharp}
Dongruo Zhou.
\newblock Sharp first-order lower bounds for higher-order smooth nonconvex optimization.
\newblock {\em arXiv preprint arXiv:2606.05438}, 2026.

\end{thebibliography}

\end{document}